\documentclass{siamltex}
\pdfoutput=1
\usepackage[bookmarks=true]{hyperref}
\usepackage{graphics}
\usepackage{mathtools}
\usepackage{algorithm}
\usepackage{bm}
\usepackage[noEnd=false,indLines=false]{algpseudocodex}
\usepackage{amsmath}
\usepackage{amsfonts}\usepackage{bm}
\usepackage{colortbl}\usepackage{multirow}\usepackage{url}
\usepackage{placeins}
\usepackage{hhline}
\usepackage{tikz}

\def\R{ {\mathbb R}}

\def\fl{\mbox{fl}}\def\sign{\mbox{sign}}
\def\eqand{\qquad\mbox{and}\qquad}

\def\nullspace{{\cal N}}
\def\vec#1{{\bm #1}}
\def\T{{\rm T}}

\let\oldenumerate=\enumerate
\renewenvironment{enumerate}{\oldenumerate\parindent=1.5em}{\endlist}

 \title{Spectral Transformation for the Dense
  Symmetric Semidefinite Generalized Eigenvalue Problem} \author{Michael
  Stewart\thanks{Department of Mathematics and Statistics, Georgia
    State University, Atlanta GA 30303, {\tt mastewart@gsu.edu}}}

\begin{document}
\maketitle
\begin{abstract}
  The spectral transformation Lanczos method for the sparse symmetric
  definite generalized eigenvalue problem for matrices $A$ and $B$ is
  an iterative method that addresses the case of semidefinite or ill
  conditioned $B$ using a shifted and inverted formulation of the
  problem. This paper proposes the same approach for dense problems
  and shows that with a shift chosen in accordance with certain
  constraints, the algorithm can conditionally ensure that every
  computed shifted and inverted eigenvalue is close to the exact
  shifted and inverted eigenvalue of a pair of matrices close to $A$
  and $B$. Under the same assumptions on the shift, the analysis of
  the algorithm for the shifted and inverted problem leads to useful
  error bounds for the original problem, including a bound that shows
  how a single shift that is of moderate size in a scaled sense can be
  chosen so that every computed generalized eigenvalue corresponds to
  a generalized eigenvalue of a pair of matrices close to $A$ and
  $B$. The computed generalized eigenvectors give a relative residual
  that depends on the distance between the corresponding generalized
  eigenvalue and the shift. If the shift is of moderate size, then
  relative residuals are small for generalized eigenvalues that are
  not much larger than the shift. Larger shifts give small relative
  residuals for generalized eigenvalues that are not much larger or
  smaller than the shift.
\end{abstract}

\section{Introduction}
\label{sec:introduction}
\pdfbookmark[0]{Introduction}{bk:introduction}

The {\em symmetric semidefinite generalized
  eigenvalue problem}
\begin{equation}
  \label{eq:sym_gen_eig}
  A \vec{v} = \lambda B \vec{v}, \qquad \vec{v}\neq \vec{0}
\end{equation}
for $n\times n$ symmetric $A$ and symmetric positive semidefinite or
positive definite $B$ arises commonly in a number of application
areas, notably in structural engineering, and has a long history in
the research literature on numerical linear algebra. When $B$ is
positive definite, the most commonly used algorithm for the dense
problem is based on Cholesky factorization of $B$. It is described in
\cite{wilk:65} and implemented in the LAPACK routine \verb|xSYGV|
\cite{abbd:99}. However, since it involves forming a matrix from the
possibly ill conditioned Cholesky factor, very little can be proven
about its stability. In practice, it typically delivers small relative
residuals for eigenvalues that are large in magnitude and larger
relative residuals for small eigenvalues. The purpose of this paper is
to describe a related algorithm that, instead of solving systems with
the Cholesky factor of $B$, solves systems using a symmetric factor of
the shifted matrix $A-\sigma B$. With an appropriately chosen shift,
the algorithm is amenable to error analysis and we provide a mixed
forward/backward error bounds for a matrix decomposition for the
shifted and inverted problem. This error analysis also provides useful
residual bounds for the original problem (\ref{eq:sym_gen_eig}). In
the remainder of the introduction, we establish notation and survey
existing methods.

We assume here that $A$ and $B$ are real and symmetric, but the
problem is not substantially different if they are complex and
Hermitian. It is necessary to assume that
\begin{equation*}
  \nullspace(A) \cap \nullspace(B) = \left\{ \vec{0} \right\},
\end{equation*}
where $\nullspace(A)$ denotes the null space of $A$, since otherwise
every choice of $\lambda$ is a solution of \eqref{eq:sym_gen_eig} for
some choice of $\vec{v}$.  If $\vec{v}$ is a null vector of $B$, then
we identify it as a generalized eigenvector for the generalized
eigenvalue $\lambda = \infty$.  We also assume throughout the paper
that $A\neq 0$ and $B\neq 0$.

Problem \eqref{eq:sym_gen_eig} can be reformulated as
\begin{equation}
  \label{eq:sym_gen_eig2}
  \beta A \vec{v} = \alpha B \vec{v}, \qquad \vec{v}\neq \vec{0},
\end{equation}
which puts $A$ and $B$ on a more equal footing, although we still
require that $B$ specifically be semidefinite. In this formulation we
refer to the pair $(\alpha, \beta)$ as a generalized eigenvalue for
the matrix pair $(A, B)$. The generalized eigenvalues from
\eqref{eq:sym_gen_eig} are given by $\lambda = \alpha/\beta$, with
infinite generalized eigenvalues having the form $(\alpha, 0)$. We
favor \eqref{eq:sym_gen_eig2} in the algorithm and in the final error
bounds, but switch between \eqref{eq:sym_gen_eig} and
\eqref{eq:sym_gen_eig2} whenever it is convenient.  Where the context
is clear we omit ``generalized'' and refer simply to eigenvalues and
eigenvectors.

The symmetric definite generalized eigenvalue problem is known to have
real eigenvalues. Applying a congruence to both $A$ and $B$ preserves
both eigenvalues and symmetry. Given a diagonalizing congruence for
which
\begin{equation*}
  V^\T A V = D_\alpha \eqand V^\T B V = D_\beta,
\end{equation*}
the eigenvalues $\lambda$ are the diagonal elements of
$D_\beta^{-t} D_\alpha$ and the eigenvectors are the columns of $V$.
Such a congruence can be shown to exist \cite{parl:80,stew:01} with a
somewhat weaker assumption than positive definiteness of $B$
\cite{govl:13,laro:05}, although positive semidefiniteness of $B$ is
not by itself sufficient. Unfortunately, even when it exists, the
diagonalizing congruence will not in general be orthogonal, which means
that we cannot expect to compute such a decomposition through the
stable application of rotations or reflectors.

The most commonly used method for positive definite $B$ computes a
diagonalizing congruence starting with the Cholesky factorization
$B = C_b C_b^\T$. We can then compute an eigenvalue decomposition
$C_b^{-\T} A C_b^{-1}= U\Lambda U^\T$ so that we have congruences
\begin{equation*}
U^\T C_b^{-1} A C_b^{-\T} U = \Lambda, \eqand 
U^\T C_b^{-1} B C_b^{-\T} U = I.
\end{equation*}
The eigenvectors are then the columns of $V = C_b^{-\T} U$. The
apparent inverses, and all references to applying inverses in this
paper, are short-hand for solving linear systems in an appropriately
stable way. There are some nontrivial aspects to exploiting symmetry
in forming $C_b^{-\T} A C_b$ that are described in \cite{stew:01}. We
refer to this as the standard method for \eqref{eq:sym_gen_eig}. It is
described more fully in \cite{stew:01, wilk:65}. It is simple and
efficient but when $B$ is ill conditioned it comes with no expectation
of stability and is in fact known to give large residuals for
eigenvalues that are of small magnitude relative to $\|\Lambda\|_2$.
If the Cholesky factorization of $B$ is computed with diagonal
pivoting and the Jacobi method is used to compute the decomposition
$U\Lambda U^\T$, then \cite{daht:01} shows that even if $B$ is
ill-conditioned, the computed eigenvalues satisfy backward error
bounds that are often much better than might be expected from
consideration of only the condition number of $B$.  The analysis can
be extended to eigenvalues computed using the $QR$ algorithm if the
initial tridiagonal reduction is performed using plane rotations.

There are a number of other options. The $QZ$ algorithm \cite{most:73}
for the nonsymmetric generalized eigenvalue problem is backward
stable. However it is significantly slower than the standard method
for \eqref{eq:sym_gen_eig} and, when applied to a symmetric problem,
the backward errors are not in general symmetric. As a consequence,
the computed eigenvalues can be complex and specialized perturbation
theory for the symmetric problem \cite{stsu:90} is not
applicable. Throwing out the potentially significant imaginary part of
a sensitive computed eigenvalue does not fix the problem in a stable
way.  The methods of \cite{fihe:72} and \cite{pewi:70b} attempt to
deflate the problem in a way that removes infinite or very large
generalized eigenvalues prior to inversion. Their stability depends on
deflating with a threshold that does not modify the problem
excessively while guaranteeing that only a well conditioned matrix
needs to be inverted to obtain the computed finite eigenvalues. This
involves balancing potentially conflicting requirements for stability
and does not always give satisfactory results.

A stable algorithm was presented in \cite{chand:00} with an error
analysis giving entirely satisfactory bounds on relative
residuals. The stability of the algorithm is unconditional, which is
ideal. However it involves computing an ordinary eigenvalue
decomposition and transformations used in deflation of an eigenvalue
can increase the size of residuals. If a test on residuals fails in
the course of the algorithm, an ordinary eigenvalue decomposition
might need to be recomputed. The number of such decompositions that
need to be computed was reported to be small in the numerical tests,
and there were arguments given as to why this should be
so. Nevertheless there is some uncertainty about how many eigenvalue
decompositions are needed.

Finally, and of particular importance to the approach taken here,
there is substantial previous work on the sparse problem. The Lanczos
algorithm can in principle be applied to $C_b^{-\T} A C_b^{-1}$ to
solve \eqref{eq:sym_gen_eig} while exploiting sparsity. However, this
poses numerical difficulties that are similar to those encountered
when computing a full eigenvalue decomposition of
$C_b^{-\T} A C_b^{-1}$. The spectral transformation Lanczos
method~\cite{erru:80, grls:94} avoids these issues by applying the
Lanczos method to $C_b^\T (A - \sigma B)^{-1} C_b$.  The merits of
this algorithm, and a variant, were further described in
\cite{npej:87}.  The main contribution of this paper is the
observation that with a suitable shift, this same transformation
conditionally stabilizes a simple direct method for
\eqref{eq:sym_gen_eig}.  In terms of stability and efficiency, the
algorithm inhabits a middle ground between the standard Cholesky based
algorithm and that of \cite{chand:00}.  Relative to \cite{chand:00},
we compromise on stability in having conditions on the shift attached
to bounds on the magnitude of residuals.  However, the algorithm is
easily implemented using factorizations already implemented in LAPACK
using level-3 BLAS operations and is much closer in cost to the
Cholesky-based method. Furthermore, it is typically possible to choose
a single moderately sized shift that will work well for all computed
eigenvalues. However for computed eigenvectors and shifts of large
magnitude, there are additional restrictions related to the magnitude
of the shift relative to the magnitudes of the eigenvalues of
interest.

In the remainder of the paper, we derive the algorithm, prove error
bounds, and present the results of numerical experiments that
illustrate the key features of the bounds. In \S\ref{sec:algorithm} we
give a full description of the algorithm, which is motivated by a
simple lemma that later provides a basis for the error analysis. We
also highlight other properties of the algorithm that will be of
importance for the analysis. A forward/backward stability result for a
matrix decomposition for the shifted and inverted problem analysis is
given in \S\ref{sec:error-analysis}. In \S\ref{sec:small_sigma0} and
\S\ref{sec:large_sigma0} we give bounds for residuals associated with
the problem \eqref{eq:sym_gen_eig2} in the distinct cases of moderate
and large shifts. The results of the numerical experiments are given
in \S\ref{sec:numerical-experiments}. Results are summarized in
\S\ref{sec:summary}.

\section{Derivation of the Algorithm}
\label{sec:algorithm}
\pdfbookmark[0]{Derivation of the Algorithm}{bk:algorithm}

We begin by choosing a shift $\sigma$ for which $A-\sigma B$ is
nonsingular and computing two decompositions
\begin{equation}
  \label{eq:initial_decomp}
  A-\sigma B = C_a D_a C_a^\T, \eqand B = C_b C_b^\T,
\end{equation}
where $D_a$ is diagonal with diagonal elements equal to $\pm 1$ and
$\|C_a\|_2$ is not much larger than $\|A-\sigma B\|_2^{1/2}$.  We
require that both decompositions have a small backward error.  The
Cholesky factorization is the obvious choice for $B$ when it is
positive definite.  However, to allow for the possibility that $B$ is
semidefinite, we allow $C_b$ to be $n\times r$ where $r\leq n$ and the
columns of $C_b$ are linearly independent.  If a rank revealing
decomposition is used with sufficiently small truncation tolerance,
then errors from truncation are comparable to backward errors from
rounding.  If $B$ is semidefinite, a pivoted Cholesky decomposition as
computed by \verb|xPSTRF| in LAPACK is efficient and is what is used
in the numerical experiments.

For $A - \sigma B$, an obvious choice is to start with a pivoted
$LDL^\T$ factorization, where $D$ has $1 \times 1$ and $2 \times 2$
blocks on the diagonal. If $A - \sigma B = P LDL^\T P^\T$, then we can
compute an eigenvalue decomposition $D = QD_\sigma D_a D_\sigma Q^\T$,
where $Q$ is orthogonal and block diagonal with $1 \times 1$ and
$2 \times 2$ diagonal blocks, $D_\sigma$ has positive diagonal
elements, and $D_a$ has $\pm 1$ on the diagonal. This decomposition of
$D$ can be computed by individually computing eigenvalue
decompositions of the $2 \times 2$ diagonal blocks, factoring out the
signs in $D_a$ and taking square roots to get $D_\sigma$. We then have
$C_a = P LQD_\sigma$ and $A - \sigma B = C_a D_a C_a^\T$. This
decomposition is stable if the $LDL^\T$ decomposition is backward
stable and the eigenvalue decompositions of the $2 \times 2$ blocks
are backward stable. There is a fine point related to the pivoting
scheme used and the magnitude of $\|C_a\|_2$. If $\|L\|_2$ can be
bounded suitably, it is easy to show that $\|C_a\|_2$ is not much
larger than $\|A - \sigma B\|_2^{1/2}$. However Bunch-Kaufman partial
pivoting \cite{buka:77} does not guarantee a bound on $\|L\|_2$
\cite{high:97,high:02}. A stronger pivoting scheme is needed to give a
useful bound on $\|C_a\|_2$ for the error analysis in
\S\ref{sec:error-analysis}. Either Bunch-Parlett complete pivoting
\cite{bupa:71} or rook pivoting \cite{asgl:98} should be used.

Given the decomposition of $B$, we can apply a shift and invert
spectral transformation as described in the following lemma, the
contents of which were used without being formally stated in
\cite{erru:80}.

\begin{lemma}
 \label{lm:spectral_transformation}
 Let $A - \sigma B$ be nonsingular and $B = C_bC_b^\T$, where $C_b$ is
 $n \times r$ and might have linearly dependent columns. Assume that
 $\lambda\neq\infty$ and $\vec{v} \neq \vec{0}$ satisfy
 \eqref{eq:sym_gen_eig}.  Then $\theta = 1/(\lambda - \sigma)$ is an eigenvalue of
 the problem 
 \begin{equation}
   \label{eq:spectral_trans_eig}
   C_b^\T (A-\sigma B)^{-1} C_b \vec{u} = \theta \vec{u}, \qquad \vec{u}\neq \vec{0}
 \end{equation}
 with eigenvector $\vec{u} = C_b^\T \vec{v}\neq \vec{0}$.  

 Conversely, assume that $\vec{u}\neq \vec{0}$ is an eigenvector for
 \eqref{eq:spectral_trans_eig} with eigenvalue $\theta$. If
 $C_b \vec{u} \neq \vec{0}$, then the vector
 $\vec{v} = (A - \sigma B)^{-1} C_b\vec{u}\neq 0$ is an eigenvector for
 \eqref{eq:sym_gen_eig2} with eigenvalue
 $(1 + \sigma \theta, \theta)$. In this case, with $\vec{v}$ defined in
 this way, we have $C_b^\T \vec{v} = \theta \vec{u}$.  If instead we
 have $C_b \vec{u} = \vec{0}$, then $\theta = 0$ and $(1, 0)$ is an
 eigenvalue for \eqref{eq:sym_gen_eig2} with eigenvector given by
 $\vec{v} = \vec{u}$. If $C_b$ is $n\times n$ and invertible, then we
 have $C_b \vec{u} \neq \vec{0}$ and can use the alternate formula
 $\vec{v} = C_b^{-\T} \vec{u}$ to obtain an eigenvector of
 \eqref{eq:sym_gen_eig2}.
\end{lemma}
\begin{proof}
  We assume that $\lambda \neq \infty$ and $\vec{v} \neq \vec{0}$
  satisfy \eqref{eq:sym_gen_eig}. If we shift \eqref{eq:sym_gen_eig},
  then we have
  $(A - \sigma B)\vec{v} = (\lambda - \sigma )C_b C_b^\T \vec{v}$. The
  fact that $\sigma$ is not an eigenvalue of $(A, B)$ ensures that
  $\lambda - \sigma \neq 0$ so that invertibility of $A - \sigma B$
  can be used to transform the shifted problem to
  \begin{equation*}
    C_b^\T \vec{v} = (\lambda - \sigma) C_b^{\T} (A-\sigma B)^{-1} C_b C_b^\T \vec{v}
  \end{equation*}
  or 
  \begin{equation*}
    \frac{1}{\lambda - \sigma} \vec{u} = C_b^{\T} (A-\sigma B)^{-1} C_b \vec{u}
  \end{equation*}
  with $\vec{u} = C_b^\T \vec{v}$. Clearly $\vec{u}\neq \vec{0}$,
  since otherwise the shifted version of \eqref{eq:sym_gen_eig} is
  $(A-\sigma B)\vec{v} = (\lambda - \sigma) C_b C_b^\T \vec{v}=0$,
  which would imply that that $\vec{v}$ is in the null space of
  $(A - \sigma B)$. Thus $1/(\lambda - \sigma)$ is an eigenvalue of
  \eqref{eq:spectral_trans_eig} as claimed.

  Now consider the second part of the lemma and assume that $\vec{u}$
  with $C_b \vec{u} \neq \vec{0}$ is an eigenvector of
  $C_b^\T (A - \sigma B)^{-1} C_b$ for eigenvalue $\theta$.  If we define
  $\vec{v} = (A - \sigma B)^{-1} C_b \vec{u}$, then $\vec{v}\neq \vec{0}$ and
  \begin{equation*}
    \theta (A-\sigma B) \vec{v} = \theta C_b \vec{u} =
    C_b C_b^\T (A-\sigma B)^{-1}C_b \vec{u}=
    B \vec{v},
  \end{equation*}
  which is equivalent to
  $\theta A \vec{v} = (1+\sigma \theta)B \vec{v}$. The claim for the
  case $C_b \vec{u} = \vec{0}$ is immediate from the fact that in this
  case $B\vec{u} = \vec{0}$. If $C_b$ is invertible then
  $\theta \neq 0$ and if we let $\vec{v} = C_b^{-\T} \vec{u}$, then we
  have
  \begin{equation*}
    \theta C_b^{-\T} \vec{u} = (A-\sigma B)^{-1}C_b \vec{u}
  \end{equation*}
  so that the two formulas are equivalent up to scaling.
\end{proof}

When solving a problem of the form \eqref{eq:spectral_trans_eig} where
$C_b$ has fewer than $n$ columns, it is useful to recall that $B$
being positive semidefinite does not guarantee a basis of $n$
generalized eigenvectors. Often $C_b^\T (A - \sigma B)^{-1} C_b$ will
be nonsingular, if possibly ill-conditioned, in which case each
shifted and inverted eigenvalue satisfies $\theta_i \neq \vec{0}$ for
$i=1,\ldots, r$ and the eigenvalues of \eqref{eq:spectral_trans_eig}
give $r$ finite eigenvalues of the form
$\lambda_i = (1 + \sigma \theta_i )/\theta_i$ with eigenvectors given
by $(A-\sigma B)^{-1} C_b \vec{u}_i$. The $n-r$ infinite eigenvalues
have eigenvectors that can be chosen as a basis for the $n - r$
dimensional null space of $B$.

If $C_b^\T (A - \sigma B)^{-1} C_b$ is singular, it is possible to
have $\theta_i = 0$, even when $C_b \vec{u}_i \neq \vec{0}$, in which
case $\lambda_i$ is infinite and the eigenvector computed from
$\vec{v}_i = (A - \sigma B)^{-1} C_b \vec{u}_i$ can be in the null space of
$B$. Consider
\begin{equation}
  \label{eq:null_vec_example}
  A =
  \begin{bmatrix}
    2 & 1 \\
    1 & 0
  \end{bmatrix}, \qquad
  B =
  \begin{bmatrix}
    1 & 1 \\
    1 & 1
  \end{bmatrix} = 
  \begin{bmatrix}
    1 \\ 1
  \end{bmatrix}
  \begin{bmatrix}
    1 & 1
  \end{bmatrix} = C_b C_b^\T.
\end{equation}
With $\sigma = 1$, it is easily seen that
$C_b^\T(A - \sigma B)^{-1}C_b = 0$ so we have $\theta = 0$ with scalar
eigenvector $\vec{u} = 1$. This gives
\begin{equation*}
  \vec{v}=(A-\sigma B)^{-1} C_b \vec{u} =
  \begin{bmatrix}
    1 \\ -1
  \end{bmatrix}
\end{equation*}
which is in the null space of $B$ and could have been identified
simply by looking at $B$ without considering $A$ at all. This is, in
fact, the only eigenvector for this matrix pair, which does not have a
complete basis of eigenvectors and is not diagonalizable.

Based on the lemma, we can solve (\ref{eq:spectral_trans_eig}) to
obtain eigenvalues and eigenvectors of \eqref{eq:sym_gen_eig2}.
Written fully in terms of the decompositions of $A - \sigma B$ and
$B$, the transformed problem \eqref{eq:spectral_trans_eig} becomes
\begin{equation*}
   C_b^\T C_a^{-T} D_a C_a^{-1} C_b \vec{u} = \theta \vec{u}.
\end{equation*}
Defining
\begin{equation}
\label{eq:XW_def}
  X \coloneqq C_a^{-1} C_b,  \eqand W \coloneqq X^\T D_a X,
\end{equation}
we have $W\vec{u} = \theta \vec{u}$.  If $U$ is the eigenvector matrix
for this problem, then the eigenvector matrix for the original problem
can be computed from
\begin{equation*}
  V = C_a^{-\T} D_a C_a^{-1} C_b U = C_a^{-\T} D_a XU,
\end{equation*}
which has linearly independent columns whenever the factor $C_b$ has
linearly independent columns.

Computing $X$ requires solving linear systems involving $C_a$, which
might be ill conditioned. Defining
\begin{equation}
  \label{eq:eta_def}
  \eta \coloneqq \frac{\|A-\sigma B\|_2^{1/2}}{\|B\|_2^{1/2}},
\end{equation}
we will show in the error analysis that this ill conditioning in
forming $X$ is harmless if $\eta\|X\|_2$ is not large. To control the
size of $\eta \|X\|_2$ we make a choice of shift guided by the
following lemma.

\begin{lemma}
  \label{lm:norm_bound}
  Let $A$ be an $n \times n$ nonzero symmetric matrix and $B$ be a
  nonzero $n \times n$ symmetric rank $r > 0$ positive semidefinite
  matrix for which $A - \sigma B$ is invertible. Let $A - \sigma B$
  and $B$ have factorizations given by \eqref{eq:initial_decomp},
  where $C_b$ has linearly independent columns. Let $X$ and $\eta$ be
  defined as in \eqref{eq:XW_def} and \eqref{eq:eta_def}, where
  $X^\T D_a X \neq 0$. Define
  \begin{equation}
    \label{eq:scaled_shift}
    \sigma_0 \coloneqq \sigma \frac{\|B\|_2}{\|A\|_2},
  \end{equation}
  and
  \begin{equation*}
    \mu \coloneqq \frac{\|X\|_2^2}{\|X^\T D_a X\|_2}.
  \end{equation*}
  Then
  \begin{equation*}
    \eta^2 \|X\|_2^2 \leq
    \left(1 + |\sigma_0|\right)
    \frac{\mu}
    {\min_i \left| \sigma_0 - \frac{\|B\|_2\lambda_i}{\|A\|_2}\right|} =
    \left(1 + \frac{1}{|\sigma_0|}\right)
    \frac{\mu}
    {\min_i \left| 1 - \frac{\lambda_i}{\sigma}\right|},
  \end{equation*}
  where each minimum is taken over all finite generalized eigenvalues
  $\lambda_i$ of $(A,B)$ and the second form of the bound assumes that $\sigma \neq 0$.
\end{lemma}
\begin{proof}
  Lemma~\ref{lm:spectral_transformation} implies that every eigenvalue
  of $\theta$ of $X^\T D_a X$ is of the form $1/|\lambda - \sigma|$
  where $\lambda$ is a generalized eigenvalue of $(A, B)$, with
  $\lambda = \infty$ if $\theta = 0$. Thus
  \begin{equation*}
    \|X^\T D_a X\|_2 = \rho\left(X^\T D_a X\right)
    \leq\frac{1}{\min_i |\lambda_i -\sigma |}
  \end{equation*}
  and
  \begin{align*}
    \eta^2 \|X\|_2^2 
    & \leq
      \frac{\|A-\sigma B\|_2}{\|B\|_2} \cdot
      \frac{\mu}
      {\min_i \left| \sigma - \frac{\lambda_i}{\sigma}\right|} \\
    & \leq
      \frac{\|A\|_2+ |\sigma_0| \|A\|_2}{\|B\|_2}
      \frac{\mu}
      {\min_i \left| \sigma - \lambda_i\right|}
      =
      (1+ |\sigma_0|)
      \frac{\mu}
      {\min_i \left| \sigma_0 - \frac{\|B\|_2 \lambda_i}{\|A\|_2}\right|}.
  \end{align*}
  The second form of the upper bound follows by factoring out $|\sigma_0|$ from the
  expression being minimized.
\end{proof}

In the proof of the lemma we have used the bound
\begin{equation}
\label{eq:shifted_A_inequality}
  \|A-\sigma B\|_2 \leq (1+|\sigma_0|) \|A\|_2,
\end{equation}
which will be of later use in the error analysis.

We refer to $\sigma_0$ as the {\em scaled shift} and
$\|B\|_2\lambda_i/\|A\|_2$ as a {\em scaled eigenvalue}.  The scaling
makes the eigenvalue and choice of shift independent of the scaling of
$A$ and $B$.  The second bound in the lemma suggests that if
$|\sigma_0|$ is not small, $\mu$ is not large, and no eigenvalue
$\lambda_i$ is too close in a relative sense to the shift $\sigma$,
then $\eta\|X\|_2$ will be of moderate size.  The first bound makes it
apparent that small $|\sigma_0|$ is not necessarily a problem if no
scaled eigenvalue is close to the scaled shift in an absolute sense.

We focus on the second bound.  Our standard for what counts as
$\lambda_i$ being “too close” to $\sigma$ is very forgiving in
practice.  We do not necessarily need to identify a substantial gap in
eigenvalues in which to place $\sigma$.  For example, suppose that we
are unlucky and the chosen $\sigma$ agrees with some $\lambda_i$ to
about 3 digits.  We might then have
$|1 - \lambda_i /\sigma | = 10^{-3}$. If $\mu = 10$ and
$\sigma_0 = 2$, then the inequality implies that
$\eta\|X\|_2\leq 122.5$. The quantity $\eta\|X\|_2$ appears as a
factor in residual bounds proven in later sections.  We might expect
to see the residual norms increase by about two orders of magnitude in
this example.  However, our numerical experiments suggest that the
bounds are pessimistic and one would likely need to make an even worse
choice of shift to see a loss of stability.  We have observed
significant problems only when computing eigenvalues in advance and
deliberately placing the shift very close to an eigenvalue.

Note that
\begin{equation*}
  \|X^\T D_a X\|_2 = \rho(X^\T D_a X) = \rho(D_aXX^\T) \leq \|X\|_2^2
\end{equation*}
so that $\mu \geq 1$. If $A - \sigma B$ is positive definite, then
$D_a = I$ and $\mu = 1$. In general we do not have direct control over
the size of $\mu$. Nevertheless, unless there is dramatic cancellation
in the product $\|X^\T D_a X\|_2$, we expect $\mu$ to be of moderate
size. Numerical experiments suggest that this is usually the case. The
problem \eqref{eq:null_vec_example} is an extreme example of exactly
the cancellation we wish to avoid.  It is easily seen that in this case
\begin{equation*}
  C_a = I, \qquad
  X =
  \begin{bmatrix}
    1 \\ 1
  \end{bmatrix}, \eqand
  D_a = A-\sigma B =
  \begin{bmatrix}
    1 & 0 \\
    0 & -1
  \end{bmatrix}
\end{equation*}
so that $X^\T D_a X = 0$. However, we also have $\eta^2 = 1/2$ so that
$\eta\|X\|_2 = 1$. We have been able to construct a problem for which
$\mu$ being large leads compromises stability, but it required some
effort. Furthermore, the cancellation that gives large $\mu$ depends on
the shift and a different choice of shift can reduce $\mu$. In
practice, the magnitude of $\eta\|X\|_2$ seems to be primarily
governed by the minimum value of $|1 - \lambda_i/\sigma |$ for most
problems and most choices of shift.

Putting everything together we have
Algorithm~\ref{alg:spectral_trans}.  In the derivation of the
algorithm, we have mentioned a number of assumptions necessary to
ensure stability of Algorithm~\ref{alg:spectral_trans}.  It is useful
to summarize them for later reference before moving on to the error
analysis.
\begin{enumerate}
\item The factorization algorithm that computes $C_a$ should guarantee
  that $\|C_a\|_2$ is not substantially larger than
  $\|A - \sigma B\|_2^{1/2}$.
\item The factorization algorithms used to compute $C_b$, $C_a$, $U$,
  and $\Theta$ should be backward stable, as should the solution of
  $C_a X = C_b$ for each individual column of $X$. The computed
  product $D_a X$ will be exact and the further product $X^\T (D_a X)$
  should be symmetric and satisfy standard error bounds for matrix
  multiplication. The multiplication $D_a X U$ and the solution of
  $C_a^\T V = D_a X U$ should be similarly stable for each column
  $\vec{v}_i$.\footnote{We include the qualification ``for each column''
    because the solution of linear systems with multiple right-hand
    sides is not in general backward stable.}
\item The shift $\sigma$ should be chosen so that $\eta\|X\|_2$ is not
  large. This is done in accordance with Lemma~\ref{lm:norm_bound} and
  if $\sigma_0$ is not too small is in most cases achieved if $\sigma$
  is not too close to an eigenvalue in a relative sense. The size of
  $\mu$ might also be a concern, but we have not observed problems in
  practice. In any event, $\mu$ is introduced solely to provide the
  bound in Lemma~\ref{lm:norm_bound} and has no impact on the error
  bounds other than in its possible effect on $\eta\|X\|_2$.
\end{enumerate}
We refer to these as the standard assumptions for
Algorithm~\ref{alg:spectral_trans}. The extent to which violating any
of these assumptions impacts the error bounds is explicit in the
bounds given in the following sections. In addition to these
assumptions, some bounds involve error terms proportional to the
magnitude of the scaled shift $\sigma_0$ , so that it is in some cases
useful to assume that this quantity is not too large. However this
should not be interpreted as a general ban on large scaled shifts. We
provide other useful residual bounds when $|\sigma_0|$ is large.

\begin{algorithm}
\caption{Spectral Transformation for \eqref{eq:sym_gen_eig}}
\label{alg:spectral_trans}
\begin{algorithmic}
\Require $A=A^\T$, $B=B^\T$, and that $B$ is positive definite or semidefinite
\Require $\eta x\_\mathrm{max} > 0$ and $\eta x\_\mathrm{max}\not\gg 1$.
\Require $\sigma\in \R$ is not too close to a generalized eigenvalue..
\Function{SpectralTransEig}{$A, B, \sigma, \eta x\_\mathrm{max}$}
\State $\breve{A} \gets A - \sigma B$
\State Factor: $\breve{A} = C_a D_a C_a^\T$
\State Factor: $B = C_b C_b^\T$
\State $\eta \gets \|A-\sigma B\|_2^{1/2} / \|B\|_2^{1/2}$
\State Solve $C_a X = C_b$ for $X$
\If{$\eta \|X\|_2 > \eta x\_\mathrm{max}$}
    \State \Return ``Error: The given $\sigma$ failed to provide a suitable bound on 
    $\eta \|X\|_2$.''
\Else
    \State $W \gets X^\T D_a X$
    \State Factor: $W = U \Theta U^\T$
    \State Solve: $C_a^{\T} V = D_a XU$ for $V$
    \State $\vec{\theta} \gets \mbox{diag}(\Theta)$
    \State $\vec{\alpha} \gets 1 \, {.+}\, \sigma \vec{\theta}$
    \State $\vec{\beta} \gets \vec{\theta}$
\EndIf
\State \Return $(V, \vec{\alpha}, \vec{\beta})$
\EndFunction
\end{algorithmic}
\end{algorithm}

Before moving on to the analysis, we note that some error bounds have
terms with
\begin{equation}
  \label{eq:gamma_def}
  \gamma \coloneqq \frac{\|A\|_2}{\|A-\sigma B\|_2}
\end{equation}
as a factor. This can be taken as a measure of cancellation of $A$
when forming $A - \sigma B$. In an arbitrary matrix norm $\|\cdot\|$
we have
\begin{equation*}
  \frac{\|A\|}{\|A-\sigma B\|}
   \leq \frac{\|A\|}{| \|A\| - |\sigma| \|B\| |}
   = \frac{1}{| 1 - |\sigma_0||},
\end{equation*}
so that $\gamma$ can be large only if $|\sigma_0|$ is close to one,
although in most cases cancellation will not occur even when
$|\sigma_0|$ is close to one.  In deriving other bounds we will
encounter $\gamma|\sigma_0|$ as a factor.  This can be bounded in
terms of $\gamma$ independently of $|\sigma_0|$.  If
$|\sigma_0| \leq 2$, then clearly $\gamma |\sigma_0| \leq 2\gamma$.
If $|\sigma_0| > 2$, then
\begin{equation*}
  \gamma |\sigma_0| 
  = |\sigma_0| \frac{\|A\|_2}{\|A-\sigma B\|_2}
  \leq \frac{|\sigma_0|}{| 1 - |\sigma_0||} =
  \frac{1}{| 1 - 1/|\sigma_0||} < 2.
\end{equation*}
Thus
\begin{equation}
  \label{eq:gamma_sigma0_bound}
  \gamma |\sigma_0| \leq 2 \max(\gamma, 1).
\end{equation}
Thus simply choosing a shift for which $|\sigma_0|$ is not too close
to one is sufficient to control the size of both $\gamma$ and
$\gamma |\sigma_0|$.  Since the common matrix norms are equivalent up
to multiplication by moderately growing functions of $n$, the choice
of the 2-norm in defining $\gamma$ and $\sigma_0$ does not have any
substantial impact on when these quantities will be large.  For
example, to control the size of $\gamma$ and $\gamma |\sigma_0|$ it is
sufficient to require that $\sigma \|B\|_\infty/ \|A\|_\infty$ not be
too close to one.  The emphasis on the matrix 2-norm in this paper,
including in Algorithm~\ref{alg:spectral_trans}, is solely for
convenience in the analysis.  It has minimal impact on the error
bounds in later sections.  In computational practice, other norms are
likely to be more convenient.

\section{Decomposition Errors}
\label{sec:error-analysis}
\pdfbookmark[0]{Decomposition Errors}{bk:decomp_errors}

Let $u$ be the unit roundoff. The error analysis is purely first order
so that we freely ignore terms that are $O(u^2)$. The notation
$\fl(\cdot)$ indicates the result of computing an expression with
roundoff errors. Unless otherwise noted, variables referring to
matrices, vectors, and scalars in Algorithm~\ref{alg:spectral_trans}
(for example $X$) refer to computed quantities. The main result of
this section is the following theorem.

\begin{theorem}
  \label{th:backward_forward_errors}
  In Algorithm~\ref{alg:spectral_trans}, assume that
  $\breve{A} = \fl(A - \sigma B)$ is computed in the obvious
  element-wise way so that $\breve{A}=\breve{A}^\T$. Assume also that
  the computed $C_a$, $C_b$, $X$, $W$, $U$, and $\theta$ satisfy
  \begin{equation}
    \label{eq:A_factor_error}
    \|\breve A - C_a D_a C_a^\T\|_2 \leq u a_n\|\breve{A}\|_2, \qquad
    \|C_a\|_2 \leq b_n \|\breve{A}\|_2^{1/2},
  \end{equation}
  \begin{equation}
    \label{eq:Bfactor_error}
    \|B- C_b C_b^\T\| \leq u c_n \|B\|_2.
  \end{equation}
  \begin{equation}
    \label{eq:Xerror}
    \left\| C_a X - C_b \right\|_2 
    \leq u d_n \|C_a\|_2 \|X\|_2
  \end{equation}
  \begin{equation}
    \label{eq:Werror}
    \|W - X^\T D_a X\|_2 \leq u e_n \|X\|_2^2, \qquad W=W^\T,
  \end{equation}
  \begin{equation}
    \label{eq:W_eig_error}
    \|W - \tilde{U} \Theta \tilde{U}^\T\|_2 \leq u f_n \|W\|_2, \eqand
    \|\tilde{U} - U\|_2 \leq u g_n,
  \end{equation}
  where $\tilde{U}$ is orthogonal and $a_n$, $b_n$, $c_n$, $d_n$,
  $e_n$, $f_n$, $e_n$,and $g_n$ are positive functions of $n$ that
  would be expected not to grow too quickly if stable algorithms are
  used for each individual computation in the algorithm. Let
  $\sigma_0$ be defined by \eqref{eq:scaled_shift}. Then there exist
  $F_1$ and symmetric $\breve{E}$, $E$, $F$ , and $G$ satisfying
  \begin{equation}
    \label{eq:error_relation_B}
    (C_b+F_1) (C_b+F_1)^\T = B + F,
  \end{equation}
  \begin{equation}
    \label{eq:breveE_factorization_errors}
    (A+E) - \sigma (B+F) = A - \sigma B + \breve{E} = C_a D_a C_a^\T
  \end{equation}
  and
  \begin{equation}
    \label{eq:error_relation_mixed}
    (C_b+F_1)^\T (A+E-\sigma (B + F))^{-1} (C_b + F_1)
    = \tilde{U}(\Theta +G)\tilde{U}^\T,
  \end{equation}
  with
  \begin{equation}
    \label{eq:F1_bound}
    \|F_1\|_2 \leq u b_n d_n \eta \|X\|_2 \|C_b\|_2 + O(u^2),
  \end{equation}
  \begin{equation}
    \label{eq:F_bound}
    \|F\|_2
    \leq u \left(c_n + 2 b_n d_n \eta\|X\|_2\right)\|B\|_2 + O(u^2),
  \end{equation}
  \begin{equation}
    \label{eq:Ebreve_bound}
    \|\breve{E}\|_2
    \leq nu \|\sigma B\|_2 + u ( n+a_n ) \|A-\sigma B\|_2 + O(u^2),
  \end{equation}
  \begin{equation}
    \label{eq:E_bound}
    \|E\|_2
    \leq u \left(n(1+2|\sigma_0|) + (1+|\sigma_0|) a_n 
      + |\sigma_0| c_n + 2|\sigma_0| b_nd_n\eta \|X\|_2\right)\|A\|_2 +O(u^2),
  \end{equation}
  and
  \begin{equation}
    \label{eq:G_bound}
    \|G\|_2 \leq u (e_n+f_n) \|X\|_2^2 + O(u^2).
  \end{equation}
\end{theorem}

The proof of the theorem is given in
Appendix~\ref{sec:proofs-theor-refth:b}.  The bounds assumed by the
theorem do not require specific choices of algorithms for the various
computations performed in Algorithm~\ref{alg:spectral_trans}. However
their form is consistent with standard bounds that can be proven when
using stable algorithms for each individual computation. Specifically,
the bounds correspond to operations that match the assumptions stated
at the end of \S\ref{sec:algorithm}.

If we consider \eqref{eq:error_relation_B} and
\eqref{eq:error_relation_mixed} and compare this result to
Lemma~\ref{lm:spectral_transformation} and the comments made
immediately following the lemma, we see that the exact nonzero
eigenvalues $\hat{\theta}_i$ of $\Theta + G$ are such that each
$(1 + \sigma \hat{\theta}_i, \hat{\theta}_i )$ is an exact generalized
eigenvalue of $(A + E, B + F )$. If $\hat{\theta}_i=0$, then
$(A + E, B + F )$ has an eigenvalue (1, 0), but how the eigenvector
might be computed is dependent on whether
$(C_b + F_1)\vec{v} = \vec{0}$, as described in
Lemma~\ref{lm:spectral_transformation}. In practice, we expect to see
computed $\theta_i$ that are small instead of exactly zero and we will
not focus on handling the case in which a computed $\theta_i$ is
exactly zero. Instead, in later sections, we give residual bounds for
the computed eigenvalues and eigenvectors that hold to first order in
$u$, even for small small $\theta_i$.

In \eqref{eq:breveE_factorization_errors} we have written the errors
in two different ways.  This is because the bound on $\|E\|_2$ depends
on the magnitude of the scaled shift $|\sigma_0|$ while the bound on
$\|\breve{E}\|_2$ does not.  In \S\ref{sec:large_sigma0}, we use the
bounds on $\|\breve{E}\|_2$ to prove a residual bound that is of use
when $|\sigma_0|$ is large.

\section{Eigenvalue Stability with Bounded $\sigma_0$}
\label{sec:small_sigma0}
\pdfbookmark[0]{Stability with Bounded Shift}{bk:bounded_shift}

We now focus on eigenvalues computed when $\sigma_0$ is of moderate
size and show that under the standard assumptions from
\S\ref{sec:algorithm}, if $\theta_i$ is the computed eigenvalue from
Algorithm~\ref{alg:spectral_trans}, then the pair
$(1 + \sigma \theta_i , \theta_i)$ is an eigenvalue of a matrix pair
close to $(A, B)$. In interpreting the significance of
Theorem~\ref{th:backward_forward_errors} for the accuracy of computed
eigenvalues, we use standard perturbation theory for the symmetric
eigenvalue problem and standard results on residuals. The following
theorem, due to Weyl, can be found in \cite{stsu:90}.
\begin{theorem}
  \label{th:perturb}
  If $\Theta$ and $\Theta+G$ are $n\times n$ and Hermitian with
  eigenvalues $\theta_i$ and $\hat{\theta}_i$ arranged in
  nondecreasing order.  Then
  \begin{equation*}
    |\theta_i - \hat{\theta}_{i}| \leq \|G\|_2, \qquad
    \mbox{for $i=1,2,\ldots, n$}.
  \end{equation*}
\end{theorem}

The following lemma relates the relative residual to norm-wise
relative backward errors on $A$ and $B$.  It is similar to backward
error results \cite{frto:98, hihi:98} except that it is stated in
terms of $\alpha_i$ and $\beta_i$ instead of $\lambda_i$.
\begin{lemma}
\label{lm:general_residual_bounds}
Suppose that for some real $\beta_i$ and $\alpha_i$, and $\vec{v}_i$ we have
\begin{equation*}
  (\beta_i A - \alpha_i B) \vec{v}_i = \vec{r}_i,
\end{equation*}
where $\alpha_i$ and $\beta_i$ are not both zero and
$\vec{v}_i\neq \vec{0}$.  If $\vec{r}_i = \vec{0}$ then
$(\alpha_i, \beta_i)$ is an exact eigenvalue for eigenvector
$\vec{v}_i$.  If $\vec{r}_i\neq \vec{0}$ and
\begin{equation}
  \label{eq:relative_residual}
  \|\vec{r}_i\|_2 \leq (|\beta_i| \|A\|_2 + |\alpha_i| \|B\|_2) \|\vec{v}_i\|_2 \epsilon
\end{equation}
for $\epsilon > 0$, then there exist $E$ and $F$ such that
\begin{equation*}
  (\beta_i (A+E) - \alpha_i (B+F)) \vec{v}_i = \vec{0},
\end{equation*}
with
\begin{equation*}
  \max\left( \frac{\|E\|_2}{\|A\|_2}, \frac{\|F\|_2}{\|B\|_2} \right) 
  \leq \epsilon.
\end{equation*}
\end{lemma}
\begin{proof}
  The matrices
  \begin{equation*}
    E = \frac{-\sign(\beta_i)\|A\|_2}
    {(|\beta_i| \|A\|_2 + |\alpha_i| \|B\|_2)\|\vec{v}_i\|_2^2} \vec{r}_i \vec{v}_i^\T
  \end{equation*}
  and  
\begin{equation*}
    F = \frac{\sign(\alpha_i)\|B\|_2}
    {(|\beta_i| \|A\|_2 + |\alpha_i| \|B\|_2)\|\vec{v}_i\|_2^2} \vec{r}_i \vec{v}_i^\T
  \end{equation*}
  satisfy the claims of the lemma.
\end{proof}

The main result of this section is the following theorem.
\begin{theorem}
  \label{th:eigenvalue_residuals}
  Let $C_a$, $C_b$ , $X$, $\eta$, and $\theta_i$ be computed
  quantities from Algorithm~\ref{alg:spectral_trans}.  Let $E$, $F$,
  $F_1$,$e_n$, and $f_n$ be as in
  Theorem~\ref{th:backward_forward_errors}.  Suppose that
  $(A + E) - \sigma(B + F)$ is invertible.  Without loss of
  generality, assume that the computed eigenvalues $\theta_i$ are in
  nondecreasing order and let $\hat{\theta}_i$ and $\hat{\vec{u}}_i$
  for $i=1,2, \ldots r$ be eigenvalues and eigenvectors of
  \begin{equation*}
    \hat{W} = (C_b+F_1)^\T (A + E - \sigma (B+F))^{-1} (C_b+F_1)
  \end{equation*}
  with the $\hat{\theta}_i$ also in nondecreasing order. 
  Assume that $(C_b + F_1 )\hat{\vec{u}}_i \neq \vec{0}$ and define
  \begin{equation*}
    \hat{\vec{v}}_i =  (A+E-\sigma(B+F))^{-1} (C_b+F_1) \hat{\vec{u}}_i\neq \vec{0}.
  \end{equation*}
  Note that formula exactly defines $\hat{\vec{v}}_i$, which should
  not be interpreted as a vector computed from this formula with
  further numerical error.  Then
  \begin{multline*}
    \left\| \big(\theta_i (A + E) - (1+\sigma \theta_i) (B + F)\big)
      \hat{\vec{v}}_i \right\|_2 \leq \\
    u (e_n + f_n) (1+|\sigma_0|) \eta^2 \|X\|_2^2
    \left(|\theta_i| \|A+E\|_2 
      + |1+\sigma\theta_i| \|B+F\|_2\right) \|\hat{\vec{v}}_i\|_2 + O(u^2).
  \end{multline*}
\end{theorem}
\begin{proof}
  From Theorem~\ref{th:backward_forward_errors} we know that for
  $i=1,2,\ldots, r$ the $\theta_i$ are the eigenvalues of $\Theta$ and
  the $\hat{\theta}_i$ are eigenvalues of $\Theta + G$.
  Theorem~\ref{th:perturb} then implies that
  \begin{equation*}
    \hat{\theta}_i = \theta_i + \delta_i,
  \end{equation*}
  where
  \begin{equation*}
    |\delta_i| \leq u (e_n+f_n) \|X\|_2^2 + O(u^2).
  \end{equation*}
  Since, in addition to having assumed that the $\hat{\theta}_i$ and
  $\hat{\vec{u}}_i$ are eigenvalues and eigenvectors of $\hat{W}$, we also
  have \eqref{eq:error_relation_B},
  Lemma~\ref{lm:spectral_transformation} is applicable and implies
  that the exact generalized eigenvector of the pair $(A+E, B+F)$ is
  $\hat{\vec{v}}_i$ so that
  \begin{equation*}
    (\theta_i+\delta_i) (A+E) \hat{\vec{v}}_i 
    = (1+\sigma (\theta_i +\delta_i)) (B+F)\hat{\vec{v}}_i,
  \end{equation*}
  or
  \begin{equation*}
    \theta_i (A+E) \hat{\vec{v}}_i + H_i \hat{\vec{v}}_i 
    = (1+\sigma \theta_i) (B+F)\hat{\vec{v}}_i,
  \end{equation*}
  where
  \begin{equation*}
    H_i = \delta_i ( A+E - \sigma (B+F))
  \end{equation*}
  with $H_i^\T = H_i$.  The residual is
  $\hat{\vec{r}}_i = -H_i\hat{\vec{v}}_i$.  We have
  \begin{equation*}
    \|H_i\|_2 \leq u (e_n+f_n) \|X\|_2^2 \|A - \sigma B\|_2 +O(u^2)
    = u (e_n+f_n) \eta^2 \|X\|_2^2 \|B\|_2 +O(u^2).
  \end{equation*}

  We consider two cases.  Either
  \begin{equation*}
    |\theta_i| \leq \frac{\|B\|_2}{(1+|\sigma_0|)\|A\|_2}\qquad\mbox{or}\qquad
    |\theta_i| > \frac{\|B\|_2}{(1+|\sigma_0|)\|A\|_2}.
  \end{equation*}
  In the first case, we have
  \begin{equation*}
    |1+\sigma \theta_i| \geq 1 - |\sigma \theta_i| \geq 1 - 
    \frac{|\sigma| \|B\|_2}{(1+|\sigma_0|)\|A\|_2} = 
    1 - \frac{|\sigma_0|}{(1+|\sigma_0|)} = \frac{1}{1+|\sigma_0|} > 0.
  \end{equation*}
  Consequently
  \begin{align*}
    \|H_i\|_2 & \leq u (e_n + f_n) \eta^2 \|X\|_2^2 
                \frac{|1+\sigma \theta_i|}{1 - |\sigma \theta_i|} \|B\|_2+ O(u^2) \\
              & \leq u  (e_n + f_n) \eta^2 \|X\|_2^2 |1+\sigma \theta_i|
                (1+|\sigma_0|) \|B\|_2 + O(u^2).
  \end{align*}
  Thus
  \begin{equation*}
    \|\hat{\vec{r}}_i\|_2 \leq \|H_i\|_2 \|\hat{\vec{v}}_i\|_2 \leq
    u (e_n + f_n) \eta^2\|X\|_2^2 (1+|\sigma_0|) |1+\sigma \theta_i| \|B+F\|_2 \|\hat{\vec{v}}_i\|_2 +O(u^2),
  \end{equation*}
  which establishes the theorem for the first case.

  In the second case, we have
  \begin{align*}
    \|H_i\|_2 & \leq u (e_n + f_n) \eta^2 \|X\|_2^2 \|B\|_2 + O(u^2) \\
              & \leq u (e_n + f_n) \eta^2 \|X\|_2^2 |\theta_i|
                \frac{(1+|\sigma_0|)\|A\|_2}{\|B\|_2}  \|B\|_2 + O(u^2) \\
              & = u (e_n + f_n) \eta^2 \|X\|_2^2 |\theta_i| (1+|\sigma_0|)\|A\|_2 + O(u^2).
  \end{align*}
  Thus
  \begin{equation*}
    \|\hat{\vec{r}}_i\|_2 \leq \|H_i\|_2 \|\hat{\vec{v}}_i\|_2 
    \leq u (e_n + f_n) \eta^2\|X\|_2^2 (1+|\sigma_0|) |\theta_i| 
    \|A+E\|_2 \|\hat{\vec{v}}_i\|_2 +O(u^2),
  \end{equation*}
  which establishes the theorem for the second case.
\end{proof}

If the standard assumptions for Algorithm~\ref{alg:spectral_trans}
from \S\ref{sec:algorithm} hold and $|\sigma_0|$ is not large, then
this theorem, Lemma~\ref{lm:general_residual_bounds}, and
Theorem~\ref{th:backward_forward_errors} imply that for the computed
$\theta_i$, each eigenvalue $(1 + \sigma\theta_i, \theta_i)$ is a
generalized eigenvalue of a pair of matrices close to $(A, B)$. This
holds regardless of the magnitude of
$\lambda_i = (1 + \sigma\theta_i)/\theta_i$. We have proven this
result by showing that $\hat{\vec{v}}_i$ achieves a small
residual. However the theorem does not say anything about the residual
achieved by the computed eigenvector $\vec{v}_i$, which we consider in
\S\ref{sec:large_sigma0}.

The assumptions that $(A + E) - \sigma(B + F)$ is invertible and that
$(C_b + F_1)\hat{\vec{u}}_i \neq \vec{0}$ are less restrictive than
they might at first seem.  The invertibility assumption is guaranteed
by \eqref{eq:breveE_factorization_errors} if the numerical
factorization of $A - \sigma B$ succeeds and results in invertible
$C_a$. This will happen if there are no zero pivots in the $LDL^\T$
factorization of $A - \sigma B$.  Furthermore, the first order analysis
has shown that any possible harm from potentially ill conditioned
$C_a$ must be a second order term.  

The assumption $(C_b + F_1)\hat{\vec{u}}_i \neq \vec{0}$ is necessary
to guarantee that $\hat{\vec{v}}_i \neq \vec{0}$. If
$(C_b + F_1)\hat{\vec{u}}_i = \vec{0}$ holds exactly, then
$\hat{W}\hat{\vec{u}}_i=\vec{0}$ so that $\hat{\theta}_i = 0$ and
$|\theta_i| \leq u(e_n + f_n )\|X\|_2^2+O(u^2) = O(u)$. If, following
Lemma~\ref{lm:spectral_transformation}, we let
$\hat{\vec{v}}_i = \hat{\vec{u}}_i$, then
\begin{equation*}
(B + F)\hat{\vec{v}}_i = (C_b + F_1) (C_b + F_1)^\T \hat{\vec{v}}_i = \vec{0}
\end{equation*}
and
\begin{multline*}
  \|(\theta_i (A+E) - (1+\sigma \theta_i)(B+F))\hat{\vec{v}}_i\|_2 =
  \|\theta_i (A+E) \hat{\vec{v}}_i\|_2 \\
  \leq u(e_n+f_n) \|X\|_2^2 \|A\|_2 \|\hat{\vec{v}}_i\|_2 + O(u^2) 
  = u (e_n+f_n) \gamma \eta^2 \|X\|_2^2 \|B\|_2^2 \|\hat{\vec{v}}_i\|_2,
\end{multline*}
where we have used the fact that $\eta^2\gamma = \|A\|_2/\|B\|_2$.
This shows that with a different definition of $\hat{\vec{v}}_i$, a
satisfactory residual bound still
holds. Lemma~\ref{lm:general_residual_bounds} then guarantees a small
backward error for the eigenvalue $(1 + \sigma\theta_i, \theta_i)$.

\section{Computed Eigenvectors and the Effect of Large $\sigma_0$}
\label{sec:large_sigma0}
\pdfbookmark[0]{Stability with Large Shift}{bk:large_shift}

In this section, we consider the residuals for the computed
eigenvectors. The errors associated with the computation of
$\vec{v}_i$ involve a larger number of terms. To make the analysis
more manageable, we use $O(u)$ to denote an error that can be bounded
by an expression of the form $h_n u$, where $h_n > 0$ depends solely
on $n$ and not on any other quantities determined by the elements of
$A$ and $B$ or by the shift.  We make one exception to this approach,
which is that we do not hide $c_n$ from \eqref{eq:Bfactor_error}
behind the $O(u)$ notation.  The reason for this is that factoring $B$
into a possibly reduced rank factorization $C_bC_b^\T$ might involve
truncation error using a tolerance that would impact the magnitude of
$c_n$.  It seems appropriate to keep track of this separately from the
other errors, which are solely due to rounding.

\begin{lemma}
  \label{lm:residual_bound}
  Let $C_a$, $C_b$, $X$, $\theta_i\neq 0$, and $U$ be computed
  quantities from Algorithm~\ref{alg:spectral_trans}. Let $\eta$ and
  $\gamma$ be as in \eqref{eq:eta_def} and \eqref{eq:gamma_def}. Let
  the errors associated with the algorithm be as in
  Theorem~\ref{th:backward_forward_errors}. If $\vec{v}_i$ is computed
  from $C_a^\T \vec{v}_i = D_a X\vec{u}_i$ using a backward stable
  algorithm for the solution of the system and standard matrix-vector
  multiplication for forming the right-hand side, then
  \begin{multline*}
    \left\| \Big(\theta_i A - (1+\sigma \theta_i) B\Big)
      \vec{v}_i \right\|_2 \leq u c_n +
    O(u)\bigg[ \eta \|X\|_2 \\
    + \left( 1 + \frac{1}{|\eta^2\theta_i|} +\max(\gamma,1) \right)
    \eta^2\|X\|_2^2
    \bigg]
    \|B\|_2 \|\vec{v}_i\|_2.
  \end{multline*}
\end{lemma}

The proof of Lemma~\ref{lm:residual_bound} is given in
Appendix~\ref{sec:proofs-theor-refth:b}.  The lemma leads to the
following residual bounds.
\begin{theorem}
  \label{th:residual_bound}
  Assume that $\sigma\neq 0$ and $\lambda_i\neq 0$.  Under the assumptions of
  Lemma~\ref{lm:residual_bound}, we have
  \begin{multline*}
    \left\| \big(\theta_i A - (1+\sigma \theta_i) B\big) \vec{v}_i \right\|_2
    \leq \\
    |1+\sigma \theta_i| \cdot |1-\sigma/\lambda_i| \Bigg[
    uc_n +O(u)\Bigg( \eta\|X\|_2 \\
    + \left(1 + \max(\gamma, 1)
      \left(1+ \left| 1 - \frac{\lambda_i}{\sigma}\right|\right)\right)
    \eta^2\|X\|_2^2 \Bigg) \Bigg] \|B\|_2 \|\vec{v}_i\|_2,
  \end{multline*}
  and
  \begin{multline*}
    \left\| \big(\theta_i A - (1+\sigma \theta_i) B\big) \vec{v}_i \right\|_2
    \leq \\
    |\theta_i| \cdot |1-\lambda_i/\sigma| \cdot |\sigma_0| \Bigg[
    uc_n +O(u)\Bigg( \eta\|X\|_2 \\
    + \left(1 + \max(\gamma, 1) \left(1+ \left| 1
          - \frac{\lambda_i}{\sigma}\right|\right)\right)
    \eta^2\|X\|_2^2 \Bigg) \Bigg] \|A\|_2 \|\vec{v}_i\|_2,
  \end{multline*}
\end{theorem}
\begin{proof}
  The first bound follows from Lemma~\ref{lm:residual_bound} using
  the identities
  \begin{equation*}
    1+\sigma\theta_i = 1 + \frac{\sigma}{\lambda_i-\sigma} = \frac{\lambda_i}{\lambda_i - \sigma}
    = \frac{1}{1 - \sigma/\lambda_i}
  \end{equation*}
  and
  \begin{multline*}
    \frac{1}{|\eta^2\theta_i|} 
    = \frac{|\lambda_i - \sigma| \cdot\|B\|_2}{\|A-\sigma B\|_2} 
    = \frac{|1-\lambda_i/\sigma| \cdot |\sigma| \cdot \|B\|_2}{\|A-\sigma B\|_2}\\
    = \frac{|1-\lambda_i/\sigma| \cdot |\sigma_0| \cdot \|A\|_2}{\|A-\sigma B\|_2}
    = |1-\lambda_i/\sigma| \cdot \gamma |\sigma_0|
    \leq 2 |1-\lambda_i/\sigma| \max(\gamma, 1).
  \end{multline*}
  We have not bothered to distinguish between exact and computed
  quantities because the quantities to be bounded are multiplied by
  $O(u)$ in the bound given by the theorem.  The second bound follows
  from the same considerations as the first upon noting that
  \begin{equation*}
    \frac{1}{|\theta_i|} \|B\|_2 = |1 - \lambda_i/\sigma| \cdot \|\sigma B\|_2
    = |1-\lambda_i / \sigma| \cdot |\sigma_0| \cdot \|A\|_2.
  \end{equation*}
\end{proof}

The first bound of the theorem suggests that under the standard
assumptions for Algorithm~\ref{alg:spectral_trans} listed at the end
of \S\ref{sec:algorithm}, the eigenvalue
$(1 + \sigma\theta_i, \theta_i)$ and computed eigenvector $\vec{v}_i$
achieve a small residual when neither $|1 - \sigma/\lambda_i |$ nor
$|1 - \lambda_i /\sigma|$ is large, $\eta \|X\|_2$ is not large, and
cancellation is avoided in forming $A-\sigma B$ so that $\gamma$ is
not large.  Typically the more important factors are
$|1 - \sigma/\lambda_i|$ and $|1 - \lambda_i/\sigma|$. These are
moderate when the eigenvalue $\lambda_i$ is not much smaller or much
larger in magnitude than the shift. For an eigenvalue $\lambda_i$ that
is close to $\sigma$ in a relative sense, then these factors become
small, which can limit the effect of larger $\eta \|X\|_2$ on the
residuals.  If the residuals are small, then
Lemma~\ref{lm:general_residual_bounds} guarantees that the computed
generalized eigenvalue and eigenvector are a generalized eigenvalue
and eigenvector for a pair of matrices close to $(A, B)$.

The second bound includes $|\sigma_0|$ as a factor and drops
$|1 - \sigma/\lambda_i|$. It implies that if the scaled shift is not
too large, and the other conditions hold, then eigenvectors associated
with eigenvalues that are not much larger than the shift achieve a
small residual, even if the eigenvalue is many orders of magnitude
smaller than the shift.  These general trends are observed in the
numerical experiments.

\section{Numerical Experiments}
\label{sec:numerical-experiments}
\pdfbookmark[0]{Numerical Experiments}{bk:num_exp}

The algorithm has been implemented in the Julia programming
language.\footnote{The code is available
  at
  \url{https://github.com/m-a-stewart/DenseSpectralTransformation.jl}.
  The specific commit used to generate the plots in this draft is
  tagged as \texttt{PaperSubmitted}.} All computations use double
precision. The pivoted Cholesky function included with Julia is based
on LAPACK \verb|xPSTRF| and is used to factor $B$ using a tolerance of
zero. The zero tolerance ensures that the algorithm continues the
factorization until a negative pivot is encountered, at which point
the factorization is truncated. If $B$ is truly close to semidefinite,
this gives a reduced rank factorization of a rank deficient matrix
close to B, guaranteeing that $\|B-C_bC_b^\T\|_2 \leq u c_n \|B\|_2$
for $c_n$ that is not large.  The implementation of the factorization
of $A- \sigma B$ uses the Julia function \verb|bunchkaufman| which,
with an appropriate option, actually computes an $LDL^\T$
factorization with rook pivoting as outlined in
\S\ref{sec:algorithm}. The function calls \verb|xSYTRF_ROOK| from
LAPACK\@. The eigenvalue decomposition is computed using the Julia
function \verb|eigen|, which calls LAPACK \verb|xSYEVR|. The solution
of $C_a X = C_b$ was done using a factorization $C_a = P LQD$ where
$Q$ is orthogonal with $1\times 1$ and $2 \times 2$ blocks on the
diagonal, $P$ is a permutation, $L$ is unit lower triangular, and $D$
is diagonal with positive elements on the diagonal. This factorization
was obtained from the $LDL^\T$ factorization of $A - \sigma B$ with
rook pivoting as described in \S\ref{sec:algorithm}.  The matrices $A$
and $B$ are based on matrices available from the NIST Matrix Market
web site at \url{https://math.nist.gov/MatrixMarket/}. The matrix $A$
is \verb|bcsstk13| from the Harwell-Boeing collection. The matrix $B$
is a modified version of \verb|bcsstm13| in which we have added
\begin{equation*}
  e^{-0.02 (n-k+1)} \|B_0\|_2
\end{equation*}
to the $k$th diagonal element of the original matrix $B_0$. This was
done because the original $B_0$ is exactly singular with multiple zero
rows and columns. The standard algorithm based on Cholesky
factorization of $B_0$ fails, which is inconvenient for making a
comparison. This particular modification also provides a range of
generalized eigenvalues without large gaps that might obscure general
trends in the dependence of the residuals on the magnitude of the
eigenvalue. Both $A$ and $B$ are $2003 \times 2003$ and
sparse. Neither is diagonal. This might justify using the spectral
transformation Lanczos method, but the matrices are small enough that
using a dense factorization method is reasonable. Both matrices are
positive definite so that the generalized eigenvalues are positive. We
have run the code with indefinite $A$ without any notable differences
in the results. We have
\begin{equation*}
  \kappa_2(A) = 1.1\times 10^{10}, \qquad \|A\|_2 = 3.1\times 10^{12}
\end{equation*}
\begin{equation*}
  \kappa_2(B) = 2.4\times 10^{17}, \eqand \|B\|_2 = 257.9.
\end{equation*}
The unmodified matrix $B_0$ and has essentially the same
norm as $B$.

To provide a basis for comparison, eigenvalues, eigenvectors, and
their residuals were computed using the standard Cholesky-based
method. The Cholesky factorization of $B$ ran successfully to
completion despite the ill conditioning of
$B$. Figure~\ref{fig:standard} is a plot of the relative residuals
versus the eigenvalues, with the left plot showing relative residuals
for computed eigenvectors with $\alpha_i = \lambda_i$ and
$\beta_i=1$. Computed negative eigenvalues are shown by negating the
eigenvalue and plotting the residual as a triangle. As expected, the
residuals are large for smaller magnitude eigenvalues.  Some of the
smaller generalized eigenvalues are negative, even though the matrices
have positive generalized eigenvalues. To assess the quality of the
eigenvalues independently of the computed eigenvectors, we note that
an inequality of the form \eqref{eq:relative_residual} holds for some
$\tilde{\vec{v}}_i$ if and only if
\begin{equation*}
  \frac{\sigma_n(\beta_i A - \alpha_iB)}{|\beta_i| \|A\|_2 + |\alpha_i| \|B\|_2}
  \leq \epsilon.
\end{equation*}
The left hand side of this inequality represents the best possible
relative residual for any choice of possible eigenvector. The plot on
the right of Figure~\ref{fig:standard} shows this best possible
relative residual, indicating the poor stability of the algorithm in
computing smaller eigenvalues, independently of the computed
eigenvectors. To keep the scale of the $y$-axis reasonable, we have
set a floor so that any residual less than $10^{-25}$ is shown as
equal to $10^{-25}$. The singular value was computed using inverse
iteration after applying the $QZ$ algorithm to quasi-triangularize both
$A$ and $B$. This can be done once to speed up inverse iteration for all
2003 computed singular values.
 
\begin{figure}
  \caption{Relative Residual vs. $\pm \lambda$, Standard Algorithm}
  \label{fig:standard}
  \vspace{1ex}
  \begin{tabular}{lr}
    \includegraphics[scale=.287]{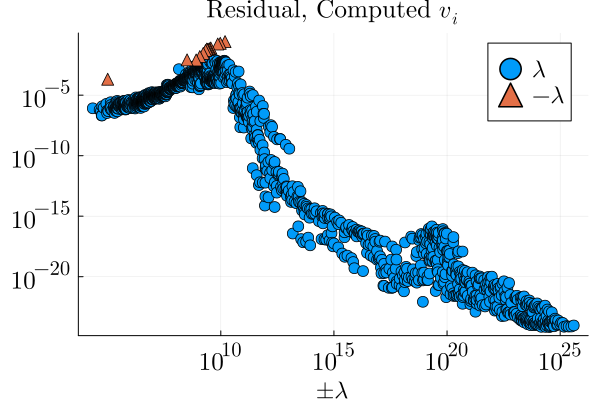} 
    & \includegraphics[scale=.287]{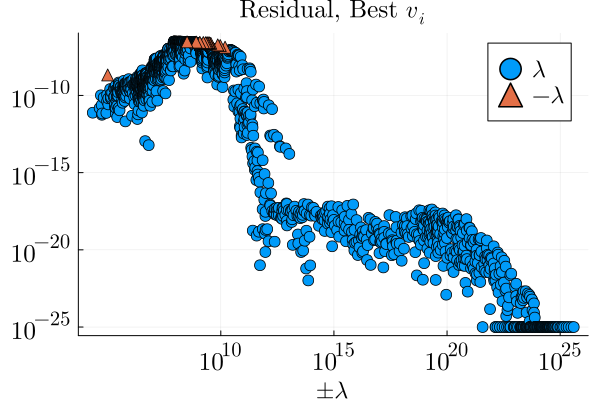}
  \end{tabular}
\end{figure}

For computed $\vec{v}_i$ , using Algorithm~\ref{alg:spectral_trans}
with a modestly sized scaled shift roughly reverses the plot for the
residuals, as expected from the analysis.  We used $\sigma_0 = 10.0$,
which corresponds to $\sigma = 1.2 \times 10^{11}$. This resulted in
$\eta \|X\|_2= 13.5$. The pivoted Cholesky ran to completion so that
2003 finite generalized eigenvalues were computed. Plots of residuals
are shown in Figure~\ref{fig:spectral_small_shift}. There are no
negative computed generalized eigenvalues. The curve is a plot of
$10^{-14} |1 - \lambda_i /\sigma|$, which is proportional to the
corresponding factor in the second inequality in
Theorem~\ref{th:residual_bound}. Residuals are small for eigenvalues
close to or smaller than $\sigma$ in magnitude. Residuals gradually
increase for larger eigenvalues, but they do not appear to increase
quite as quickly as might be expected from the theorem. Even the
largest eigenvalues have smaller residuals than do the small ones when
computed using the standard Cholesky-based algorithm. The plot for the
best possible residuals shows that for some choice of $\tilde{v}_i$,
every computed eigenvalue can result in a small residual. This is
consistent with the bounds of
Theorem~\ref{th:eigenvalue_residuals}. Every individual computed
$\theta_i$ gives an eigenvalue $(1+\sigma \theta_i , \theta_i)$ that
is an eigenvalue of a pair close to $(A, B)$.

Finally, we consider the effect of using a large scaled shift. In this
experiment, we used $\sigma_0 = 10^7$, which corresponded to
$\sigma = 1.2 \times 10^{17}$. This resulted in $\eta \|X\|_2 = 10.5$. The
pivoted Cholesky again ran to completion so that 2003 finite
generalized eigenvalues were computed. Plots of residuals are shown in
Figure~\ref{fig:spectral_large_shift}. In this case, the computed
eigenvectors and eigenvalues achieve small residuals when $\lambda_i$
is not too much larger or smaller in magnitude than the shift. The
curve is $10^{-15} |(1-\lambda_i/\sigma)(1-\sigma/\lambda_i)|$, which
appears as a factor in one of the error terms in the second bound in
Theorem~\ref{th:residual_bound}.  The second graph of the figure shows
that the eigenvalues that are orders of magnitude smaller than the
shift cannot achieve small residuals for any choice of
$\tilde{\vec{v}}_i$ and do not correspond to eigenvalues of a pair
close to $(A, B)$.

It is perhaps interesting to note that between the large shift and the
small shift, the algorithm gives generalized eigenvalues and
eigenvectors that achieve small residuals across the full range of
eigenvalues.  In practice, on this an other examples, it appears that
one can often obtain satisfactory residuals for eigenvalues over a
very large range of magnitudes using repeated applications of the
algorithm with a relatively small number of shifts.  However, success
with a very small number of shifts does seem to depend on a tendency
for the computed results to be smaller than the error bounds predict.
In the bounds, the quantities multiplying $|1 - \lambda_i /\sigma|$
and $|(1-\lambda_i/\sigma)(1-\sigma/\lambda_i)|$ are somewhat larger
than the $10^{-14}$ and $10^{-15}$ illustrated by the curves in the
figure.
 
\begin{figure}
  \caption{Relative Residual vs. $\pm \lambda$, Moderate Scaled Shift}
  \label{fig:spectral_small_shift}
  \vspace{1ex}
  \begin{tabular}{lr}
    \includegraphics[scale=.287]{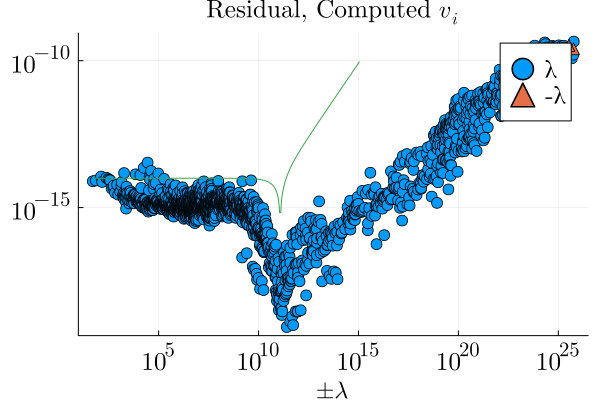} 
    & \includegraphics[scale=.287]{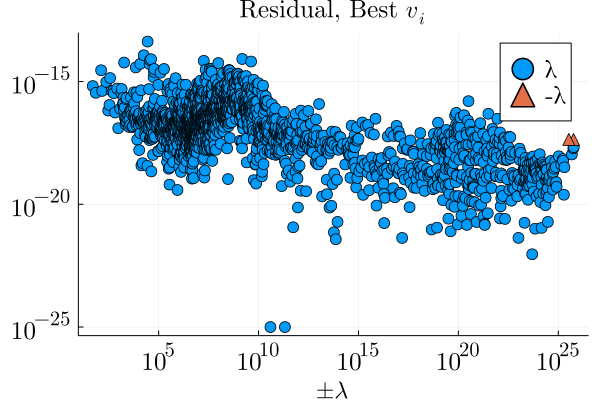}
  \end{tabular}
\end{figure}

\begin{figure}
  \caption{Relative Residual vs. $\pm \lambda$, Large Scaled Shift}
  \label{fig:spectral_large_shift}
  \vspace{1ex}
  \begin{tabular}{lr}
    \includegraphics[scale=.287]{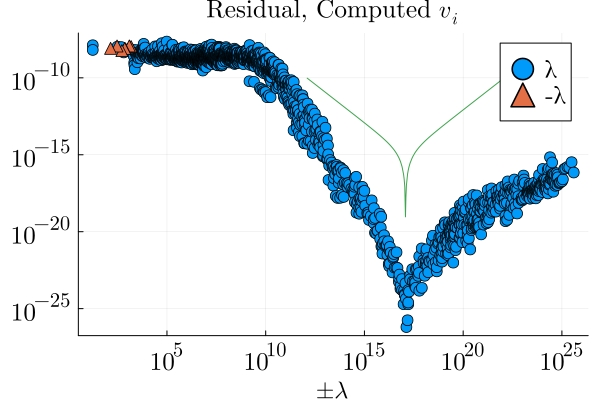} 
    & \includegraphics[scale=.287]{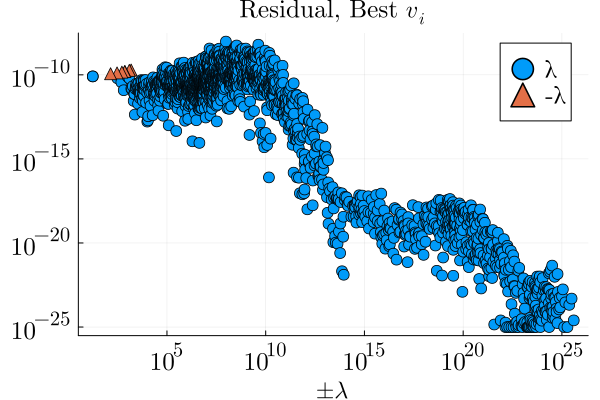}
  \end{tabular}
\end{figure}

\section{Summary}
\label{sec:summary}
\pdfbookmark[0]{Summary}{bk:summary}

The new algorithm has a number of advantages: It is efficient, with
the dominant cost being the computation of a single symmetric
eigenvalue decomposition, as with the standard Cholesky-based
method. The algorithm has useful error bounds when $\eta \|X\|_2$ is
not large and the computation of $A-\sigma B$ does not involve
cancellation resulting in large $\gamma$. In this case, with a
modestly sized scaled shift, for each computed $\theta_i$ each
eigenvalue $(1 + \sigma\theta_i , \theta_i)$ is the eigenvalue of a
pair of matrices close to $(A, B)$. Also with a modestly sized scaled
shift, computed eigenvectors for eigenvalues that are not much larger
than $\sigma$ in magnitude achieve small residuals so that the
algorithm gives an eigenvalue and eigenvector for a nearby matrix
pair.  Larger shifts can be used to focus on specific parts of the
spectrum with useful error bounds on the residuals for eigenvalues not
too far from the shift. The algorithm depends on decompositions with
efficient LAPACK implementations so that it is easy to exploit level-3
BLAS operations. Since it is based on the same transformation as the
spectral transformation Lanczos algorithm, it handles the case of
semi-definite $B$ in the same way with no difficulty.

The primary weakness of the algorithm is perhaps the need to choose a
shift.  Choosing parameters in this way is not typical for direct
methods, for which it is generally preferable to treat methods as
black boxes requiring no input from the user other than the
matrices. The requirement to choose a shift is inconvenient, but we
have not found it to be especially difficult to choose a shift that
achieves satisfactory results. Furthermore, in using the algorithm on
problems in which $\eta \|X\|_2$ is larger than might seem ideal, the
residuals do not seem to increase as quickly as the error bounds might
suggest.

The algorithm is particularly appealing in the case of positive
definite or semidefinite $A$. In this case, all generalized
eigenvalues are nonnegative and $\mu=1$. Lemma~\ref{lm:norm_bound}
implies that simply choosing a negative scaled shift that is not too
small in magnitude will ensure that $\eta \|X\|_2$ is not large. In
this case, one could simply choose $\sigma_0 = -2$, which avoids any
possibility of harmful cancellation in forming $A - \sigma B$. Each
computed eigenvalue then corresponds to an eigenvalue of a pair that
is close to $(A, B)$ and the only increase in residuals for computed
$\vec{v}_i$ is due to the factor $|1 - \lambda_i /\sigma|$. Residual
bounds for large magnitude negative shifts also benefit from the
simplifications for positive definite or semidefinite $A$.

\appendix
\section{Proofs of Theorem~\ref{th:backward_forward_errors} and
  Lemma~\ref{lm:residual_bound}}
\label{sec:proofs-theor-refth:b}

\begin{proof}[{\bf Proof of Theorem~\ref{th:backward_forward_errors}}]
Given $\sigma$, $A$, and $B$ we have
\begin{equation*}
  \breve A = \fl(A-\sigma B) = A - \sigma B + E_0
\end{equation*}
where
\begin{equation*}
  |E_0| \leq u (|A - \sigma B| + |\sigma B|) +O(u^2),
\end{equation*}
which implies
\begin{equation*}
  \|E_0\|_2 \leq u (\left\||A-\sigma B|\right\|_2 
  + \left\| |\sigma B| \right\|_2) + O(u^2).
\end{equation*}
The symmetry of $A$ and $B$ ensure that $E_0 = E_0^T$.  For a general
$n\times n$ matrix $M$ we have
\begin{equation*}
  \left\| |M| \right\|_2 \leq \sqrt{n} \||M|\|_\infty 
  = \sqrt{n} \|M\|_\infty \leq n \|M\|_2,
\end{equation*}
so that the bound can also be written as
\begin{equation*}
  \|E_0\|_2 \leq un (\left\|A-\sigma B\right\|_2 + \left\| \sigma B \right\|_2) + O(u^2).
\end{equation*}

In the statement of the theorem we have assumed
\eqref{eq:A_factor_error} and that $\breve{A}= \breve{A}^\T$.  Thus
the factorization of $\breve{A}$ satisfies
\begin{equation*}
  \breve{A} + E_1 = C_a D_a C_a^\T
\end{equation*}
where $D_a$ is diagonal with $\pm 1$ on the diagonal,
\begin{equation*}
E_1 = E_1^\T, \qquad \|E_1\|_2 \leq u a_n \|\breve{A}\|_2, \eqand
\|C_a\|_2 \leq b_n \|\breve{A}\|_2^{1/2}.
\end{equation*}
We now have a factorization of the same form as given by the second
identity of \eqref{eq:breveE_factorization_errors}, with
$\breve{E} = E_0 + E_1$, $\breve{E} = \breve{E}^\T$, and
\begin{align}
  \|\breve{E}\|_2 
  & \leq \|E_0\|_2 + \|E_1\|_2 \nonumber\\
  & \leq un(\|A-\sigma B\|_2 + \|\sigma B\|_2) + u a_n \|A-\sigma B\|_2 
    + O(u^2) \nonumber\\
  & \leq un(1+2|\sigma_0|) \|A\|_2 + u(1+|\sigma_0|) a_n \|A\|_2,
    \label{eq:Ebreve_bound_deriv}
\end{align}
where we have used \eqref{eq:scaled_shift} and
\eqref{eq:shifted_A_inequality} in the last line.  The second line is
\eqref{eq:Ebreve_bound}.

The bound \eqref{eq:Bfactor_error} implies that $B + F_0 = C_b C_b^\T$
with $\|F_0\|_2 \leq uc_n \|B\|_2$.  If we define $F_1$ by
$C_aX = C_b + F_1$ so that $X = C_a^{-1} (C_b+F_1)$, then from
\eqref{eq:Xerror} we have
\begin{multline*}
  \|F_1\|_2 \leq u d_n \|C_a\|_2 \|X\|_2 
  \leq u b_n d_n \|A-\sigma B\|_2^{1/2} \|X\|_2 + O(u^2) \\
  \leq u b_n d_n \eta \|X\|_2 \|B\|_2^{1/2} + O(u^2) 
  = ub_nd_n \eta \|X\|_2 \|C_b\|_2 +O(u^2)
\end{multline*}
and
\begin{equation*}
  (C_b+F_1)(C_b+F_1)^\T = C_b C_b^\T + F_1 C_b^\T + C_bF_1^\T + O(u^2) 
  = B+F_0 + F_1 C_b^\T + C_b F_1^\T+O(u^2).
\end{equation*}
Defining $F = F_0 + F_1C_b^\T + C_bF_1^\T$ we see that
\begin{equation*}
  \|F\|_2 \leq \|F_0\|_2 + 2 \|F_1\|_2 \|C_b\|_2 
  \leq u (c_n + 2b_n d_n\eta \|X\|_2) \|B\|_2 + O(u^2),
\end{equation*}
where we have used \eqref{eq:F_bound}.  This gives
\eqref{eq:error_relation_B} and \eqref{eq:F1_bound}.

To obtain $E$ and the bound \eqref{eq:E_bound}, we observe that
\begin{equation*}
  C_a D_a C_a^\T = A + \breve{E} + \sigma F - \sigma (B+F).
\end{equation*}
Again applying \eqref{eq:scaled_shift} we have
\begin{equation*}
  \|\sigma F\|_2 = |\sigma| \|B\|_2 \frac{\|F\|_2}{\|B\|_2}
  = |\sigma_0| \frac{\|F\|_2}{\|B\|_2} \|A\|_2.
\end{equation*}
If we define $E = \breve{E}+\sigma F$, then this identity combined
with \eqref{eq:Ebreve_bound} gives the first identity in
\eqref{eq:breveE_factorization_errors} and \eqref{eq:E_bound}.

Finally, we consider the eigenvalue decomposition. We define
$G_0 = X^\T D_a X - W$ and $G_1 = W - \tilde{U} \Theta \tilde{U}^\T$
and use \eqref{eq:Werror}, \eqref{eq:W_eig_error}, and
\eqref{eq:Ebreve_bound_deriv} to obtain
\begin{multline*}
  (C_b+F_1)^\T (A-\sigma B + \breve{E})^{-1} (C_b+F_1) 
  = (C_b+F_1)^\T C_a^{-\T} D_a C_a^{-1}(C_b + F_1) \\
  = X^\T D_a X = W + G_0 = \tilde{U}\Theta \tilde{U}^\T + G_0 + G_1.
\end{multline*}
Defining $G=G_0 + G_1$ gives
\begin{equation*}
  \|G\|_2 \leq ue_n \|X\|_2^2 + u f_n \|W\|_2 \leq u (e_n + f_n) \|X\|_2^2 + O(u^2).
\end{equation*}
Considering the definition of $E$ and $F$ and
\eqref{eq:breveE_factorization_errors}, this establishes
\eqref{eq:error_relation_mixed}.
\end{proof}

\begin{proof}[{\bf Proof of Lemma~\ref{lm:residual_bound}}]
  From \eqref{eq:error_relation_mixed} and the first identity of
  \eqref{eq:breveE_factorization_errors} we have
  \begin{equation*}
    (C_b+F_1)^\T \left(A-\sigma B + \breve{E} \right)^{-1} (C_b+F_1) \tilde{\vec{u}}_i = 
    \theta \tilde{\vec{u}}_i + \tilde{\vec{g}}_i,
  \end{equation*}
  where $\tilde{\vec{g}}_i$ is column $i$ of $G$.  Since
  $\|\tilde{\vec{u}}_i\|_2=1$ and by \eqref{eq:W_eig_error}
  $\tilde{\vec{u}}_i = \vec{u}_i+O(u)$, we can conclude from
  \eqref{eq:G_bound} that the computed $\vec{u}_i$ satisfies
  \begin{equation}
    \label{eq:u_residual}
    (C_b+F_1)^\T \left(A-\sigma B + \breve{E} \right)^{-1} (C_b+F_1) \vec{u}_i 
    = \theta_i \vec{u}_i  + \vec{g}_i,
  \end{equation}
  where
  \begin{equation*}
    \|\vec{g}_i\|_2 \leq O(u) \|X\|_2^2 + O(u^2).
  \end{equation*}
  For the computation of $\vec{v}_i$ from $\vec{u}_i$ the stability
  assumptions of the lemma imply that
  \begin{equation}
    \label{eq:V_computation_errors}
    (C_a^\T + J_1) \vec{v}_i = D_a (X+J_2) \vec{u}_i,
  \end{equation}
  where
  \begin{equation*}
    \|J_1\|_2 \leq O(u) \|C_a\|_2, \eqand
    \|J_2\|_2 \leq O(u) \|X\|_2.
  \end{equation*}
  We then have
  \begin{equation}
    \label{eq:residual0}
    C_a^\T \vec{v}_i = D_a (X+J_2) \vec{u}_i - J_1 \vec{v}_i.
  \end{equation}
  Let $C_aX = C_b+F_1$ as in the proof of
  Theorem~\ref{th:backward_forward_errors}.  Multiplying both sides of
  \eqref{eq:residual0} by $C_aD_a$ and using
  \eqref{eq:breveE_factorization_errors} gives
  \begin{equation}
    \label{eq:residual1}
    (A-\sigma B + \breve{E}) \vec{v}_i 
    = (C_b + F_1 + C_a J_2)\vec{u}_i - C_a D_a J_1 \vec{v}_i.
  \end{equation}
  Multiplying \eqref{eq:u_residual} by $C_b + F_1$ and using
  \eqref{eq:error_relation_B} and
  \eqref{eq:breveE_factorization_errors} results in
  \begin{equation*}
    \theta_i (C_b + F_1)\vec{u}_i 
    = (B+F)\left(A-\sigma B + \breve{E}\right)^{-1} (C_b + F_1)\vec{u}_i
    - (C_b + F_1)\vec{g}_i
  \end{equation*}
  so that \eqref{eq:residual1} becomes
  \begin{multline}
    \label{eq:residual2}
    \theta_i(A-\sigma B + \breve{E})\vec{v}_i =
    (B+F)\left(A-\sigma B + \breve{E}\right)^{-1}(C_b+F_1)\vec{u}_i \\
    - (C_b+F_1)\vec{g}_i + \theta_i C_a J_2 \vec{u}_i - \theta_i C_a D_a J_1 \vec{v}_i.
  \end{multline}
  Combining \eqref{eq:breveE_factorization_errors}, 
  \eqref{eq:V_computation_errors}, and $C_a X = C_b +F_1$ gives
  \begin{align*}
    (A-\sigma B + \breve{E})^{-1} (C_b + F_1) \vec{u}_i 
    & =  C_a^{-\T} D_a C_a^{-1} (C_b + F_1) \vec{u}_i \\
    & = C_a^{-\T} D_a X \vec{u}_i \\
    & = \vec{v}_i + C_a^{-\T} J_1 \vec{v}_i - C_a^{-\T} D_a J_2 \vec{u}_i.
  \end{align*}
  Thus \eqref{eq:residual2} becomes
  \begin{multline}
    \label{eq:residual3}
    \theta_i (A-\sigma B + \breve{E}) \vec{v}_i = (B+F) \vec{v}_i + BC_a^{-\T} J_1 \vec{v}_i
    - B C_a^{-\T} D_a J_2 \vec{u}_i - C_b \vec{g}_i \\
    + \theta_i C_a J_2 \vec{u}_i - \theta_i C_a D_a J_1 \vec{v}_i +O(u^2).
  \end{multline}
  Using the fact that when $\vec{v}_i$ is computed as in
  Algorithm~\ref{alg:spectral_trans},
  Lemma~\ref{lm:spectral_transformation} implies that
  \begin{equation*}
    C_b^\T \vec{v}_i = \theta_i \vec{u}_i + O(u)
  \end{equation*}
  and the fact that $BC_a^{-\T}= C_b X^\T + O(u)$ turns
  \eqref{eq:residual3} into
  \begin{multline}
    \label{eq:residual4}
    \theta_i A \vec{v}_i - (1+\sigma \theta_i) B \vec{v}_i = F \vec{v}_i - \theta_i \breve{E} \vec{v}_i 
    + C_b X^\T J_1 \vec{v}_i - \frac{1}{\theta_i} C_b X^\T D_a J_2 C_b^\T \vec{v}_i - C_b \vec{g}_i \\
    + C_a J_2 C_b^\T \vec{v}_i - \theta_i C_a D_a J_1 \vec{v}_i + O(u^2).
  \end{multline}
  The proof of the theorem then reduces to showing that each of the
  terms in \eqref{eq:residual4} is consistent with the bound in the
  theorem.

  Using the bounds in Theorem~\ref{th:backward_forward_errors} we have
  \begin{equation*}
    \|\theta_i \breve{E}\|_2 
    \leq O(u) \left(|\theta_i\sigma| \|B\|_2
      + |\theta_i| \|A-\sigma B\|_2\right)
    \leq O(u) \left(|\theta_i\sigma| + \eta^2 \|X\|_2^2 \right) \|B\|_2,
  \end{equation*}
  where
  \begin{equation*}
    |\sigma \theta_i| 
    \leq \frac{|\sigma|\|B\|_2}{\|A-\sigma B\|_2} \eta^2 \|X\|_2^2
    = \gamma |\sigma_0| \eta^2 \|X\|_2^2
    \leq 2 \max(\gamma, 1) \eta^2 \|X\|_2^2,
  \end{equation*}
  where we have used \eqref{eq:gamma_sigma0_bound}.  For the other
  terms we have
  \begin{equation*}
    \|F\|_2 \leq \left( uc_n + O(u) \eta \|X\|_2 \right)\|B\|_2,
  \end{equation*}
  \begin{equation*}
    \|C_b X^\T J_1\|_2
    \leq O(u) \|B\|_2^{1/2} \|X\|_2 \|A-\sigma B\|_2^{1/2}
    \leq O(u) \eta \|X\|_2 \|B\|_2,
  \end{equation*}
  \begin{equation*}
    \|C_b\vec{g}_i\|_2 
    = O(u) \|B\|_2^{1/2} \|X\|_2^2  
    = O(u) \frac{1}{\eta^2|\theta_i|} \|B\|_2 \eta^2 \|X\|_2^2 \|\vec{v}_i\|_2,
  \end{equation*}
  \begin{equation*}
    \left\| \frac{1}{\theta_i} C_b X^\T D_a J_2 C_b^\T \right\|_2 
    = O(u) \frac{1}{\eta^2|\theta_i|} \|B\|_2 \eta^2 \|X\|_2^2,
  \end{equation*}
  \begin{equation*}
    \|C_a J_2 C_b^\T\|_2 = O(u) \|A-\sigma B\|_2^{1/2} \|B\|_2^{1/2} \|X\|_2 = O(u)\eta \|X\|_2 \|B\|_2,
  \end{equation*}
  and
  \begin{equation*}
    \|\theta_i C_a D_a J_1\|_2 = O(u) |\theta_i| \|A-\sigma B\|_2 = O(u)\eta^2 \|X\|_2^2 \|B\|_2.
  \end{equation*}
\end{proof}

\bibliography{/home/mas/work/bib/ref}

\end{document}